\documentclass[12pt,reqno]{amsart}
\usepackage{amsmath,amssymb}

\usepackage{hyperref}

\newcommand{\R}{\mathbb{R}}

\DeclareMathOperator{\dist}{dist}
\DeclareMathOperator{\diverg}{div}
\DeclareMathOperator{\diam}{diam}

\newtheorem{theorem}{Theorem}[section]
\newtheorem{lemma}[theorem]{Lemma}
\newtheorem{corollary}[theorem]{Corollary}
\newtheorem{proposition}[theorem]{Proposition}

\numberwithin{equation}{section}

\begin{document}

\title[Boundedness of stable solutions]
{A new proof of the boundedness results for stable solutions to semilinear elliptic equations}

\author[X. Cabr\'e]{Xavier Cabr\'e}
\address{﻿X.C.\textsuperscript{1,2,3} ---
\textsuperscript{1}ICREA, Pg.\ Lluis Companys 23, 08010 Barcelona, Spain \& 
\textsuperscript{2}Universitat Polit\`ecnica de Catalunya, Departament de Matem\`{a}tiques, 
Diagonal 647, 08028 Barcelona, Spain \& 
\textsuperscript{3}BGSMath, Campus de Bellaterra, Edifici C, 08193 Bellaterra, Spain.
}
\email{xavier.cabre@upc.edu}

\thanks{X.C. is supported by grants MTM2017-84214-C2-1-P and MdM-2014-0445 (Government of Spain), 
and is a member of the research group 2017SGR1392 (Government of Catalonia).}

\begin{abstract}
We consider the class of stable solutions
to semilinear equations $-\Delta u=f(u)$ in a bounded smooth
domain of $\mathbb{R}^n$. Since 2010 an interior a priori  $L^\infty$ bound for stable solutions
is known to hold in dimensions $n\le 4$ for all $C^1$ nonlinearities~$f$. In the radial case, the same is true for $n\leq 9$.
Here we provide with a new, simpler, and unified proof of these results. It establishes, in addition, 
some new estimates in higher dimensions ---for instance $L^p$ bounds for every finite~$p$ in dimension 5.

Since the mid nineties, the existence of an $L^\infty$ bound holding for
all $C^1$ nonlinearities when $5\leq n\leq 9$ was a challenging open problem. 
This has been recently solved by A. Figalli, X. Ros-Oton, J. Serra, and the author,
for nonnegative nonlinearities, in a forthcoming paper.
\end{abstract}

\maketitle
\vspace {-.2in}
\begin{center}
{\em Dedicated to Luis Caffarelli, with friendship and great admiration}
\end{center}

\section{Introduction and results}

We consider the semilinear elliptic equation
\begin{equation}\label{eqsem}
-\Delta u = f(u) \quad\mbox{in $\Omega$,}
\end{equation}
as well as the associated Dirichlet problem
\begin{equation}\label{problem}
\left\{
\begin{array}{rcll}
-\Delta u &=& f(u) & \quad\mbox{in $\Omega$}\\
u &=& 0  & \quad\mbox{on $\partial\Omega$,}\\
\end{array}\right.
\end{equation}
where $\Omega \subset \R^n$ is a smooth bounded domain and $f$ is a $C^1$ nonlinearity.
A smooth solution $u$ of \eqref{eqsem}, or of \eqref{problem}, is said to be \emph{stable} if 
\begin{equation} \label{stability}
\int_{\Omega} f'(u)\xi^2\,\,dx
\leq
\int_{\Omega} |\nabla\xi|^2\,\,dx
\end{equation}
for all $C^1(\overline\Omega)$ functions $\xi$ such that $\xi_{|\partial\Omega}\equiv 0$. 
This is equivalent to assuming the nonnegativeness of the first 
Dirichlet eigenvalue in~$\Omega$ for the linearized operator $-\Delta-f'(u)$ (or second variation
of energy) at $u$. 
As a consequence, \emph{local minimizers} of the energy (i.e., minimizers under small perturbations
having same boundary values as $u$) are stable
solutions. 

Below we will present another interesting class of solutions ---the so called extremal solutions--- 
which provides with many examples of stable solutions. 

This article is concerned with the existence of a priori $L^\infty$ bounds for stable solutions 
of \eqref{eqsem}, or of \eqref{problem}, that hold for all $C^1$ nonlinearities $f$. 
It is Ha\"{\i}m Brezis who stressed, since the mid-nineties, the significance of obtaining 
$L^\infty$ bounds for stable solutions to
problem \eqref{problemla} below with $f$ satisfying \eqref{eq1.30}.

It is well known that $L^\infty$ estimates can only hold in dimensions $n\leq 9$. Indeed, when 
\begin{equation}
n\geq 10, \quad u=-2\log\vert{x}\vert, \quad \text{and}\quad f(u)=2(n-2)e^u,
\end{equation}
we are in the presence of a singular $H^1_0(B_1)$ stable weak solution of \eqref{problem} in $\Omega=B_1$. 
On the other hand, the results of
Capella and the author~\cite{CC} show that, in the radial case, 
such example can not exist for $n\leq 9$. More precisely, in 2006, \cite{CC} established the boundedness 
of any $H^1_0(B_1)$ stable weak solution of \eqref{problem} in the unit ball, 
for every $C^1$ nonlinearity $f$, whenever $n\leq 9$. 
The result gave also an $L^\infty(B_1)$ a priori bound with constants independent of $f$.

In the nonradial case, instead, an $L^\infty$ bound holding for all nonlinearities was known to hold only
in dimensions $n\leq 4$. This was established in 2010 by the author~\cite{Cabre}. Since the mid nineties, 
the existence of an $L^\infty$ bound holding for
all smooth nonlinearities when $5\leq n\leq 9$ was an open problem 
that has been very recently solved by Figalli, Ros-Oton, Serra, and the author \cite{CFRS}, for nonnegative nonlinearities.

The purpose of the current article is to provide a new and simpler
proof of the $L^{\infty}$ bound up to dimension 4 from \cite{Cabre}. It no longer relies on
the Michael-Simon and Allard Sobolev inequality ---a quite deep result that was applied
on every level set of the solution in \cite{Cabre}--- but instead on a new weighted Hardy inequality
for functions defined on hypersurfaces. It will be applied to the hypersurfaces given by the level sets of the solution.
Besides being simpler, the new proof is of 
interest also by the following features that the proof in \cite{Cabre} did not have, namely:

\begin{itemize}
 \item The new proof gives also the bound up to dimension 9 in the radial case 
---in a unified way with the one for $n\leq 4$ in general domains. 

 \item It provides with estimates in dimensions $n\geq 5$, such as \eqref{estgrad} below under assumption
 \eqref{condalpha}. This estimate leads to $L^p$ bounds for every finite $p$ in dimension 5. It also
gives $L^p$ estimates in higher dimensions which, however, are not optimal except for the radial case.  

 \item The new proof allows for first time to consider stable solutions of equation \eqref{eqsem} 
 independently of their boundary values. This will also be the case in the forthcoming article~\cite{CFRS}.

 \item Finally, the new proof allows for extensions to other nonlinear elliptic problems. This will be done
 in the forthcoming works~\cite{CS-P,Mir}.
\end{itemize}

\subsection{Available results}
Let us briefly comment on the main known $L^{\infty}$ bounds for stable solutions to problem \eqref{problem}. 
For more details, we refer to the book~\cite{Dupaigne}
by Dupaigne, and to the recent surveys \cite{CabSurv, CabPog} by the author.

There is a large literature on a priori estimates for stable solutions, beginning with the seminal
work of Crandall and Rabinowitz \cite{CR}. In \cite{CR} and subsequent works,
a standard assumption is that $u$ is positive in $\Omega$ and that the nonlinearity satisfies
\begin{equation} \label{eq1.30}
f(0) > 0, \textrm{ $f$ is nondecreasing, convex, and }\lim_{t\to +\infty}\frac{f(t)}t=\infty.
\end{equation}
This assumption provides with examples of stable solutions $u$ satisfying $u>0$ in $\Omega$. Indeed,
introduce now a parameter $\lambda \geq 0$ and consider the problem
\begin{equation}\label{problemla}
\left\{
\begin{array}{rcll}
-\Delta u &=& \lambda f(u) & \quad\mbox{in $\Omega$}\\
u &=& 0  & \quad\mbox{on $\partial \Omega$.}\\
\end{array}\right.
\end{equation}
Then, assuming that $f\in C^\infty$ satisfies \eqref{eq1.30}, there exists an
extremal parameter $\lambda^*\in (0,+\infty)$ such that if $0<\lambda<\lambda^*$ then
\eqref{problemla} admits a smooth positive stable solution $u_\lambda$.
This family is increasing in $\lambda$ and its limit as 
$\lambda\uparrow\lambda^{*}$ is a stable weak solution $u^*=u_{\lambda^*}$ 
of \eqref{problemla} for $\lambda=\lambda^*$. It is called \emph{the extremal solution} of \eqref{problemla}
and may be bounded or unbounded depending on the dimension, domain, and nonlinearity.
In \cite{BV}, Brezis and V\'azquez  raised several open questions ---see also the open
problems raised by Brezis in \cite{Brezis}--- 
regarding the extremal solution, in particular about its boundedness.
Next we explain the available results on this question. 

The following results from \cite{CR,Sanchon,Nedev} were proven for the {\em smooth} stable 
solutions of \eqref{problemla} corresponding to $\lambda < \lambda^*$.
They led, by letting $\lambda\uparrow\lambda^*$, to the boundedness of the extremal solution
$u^*$ under different hypothesis on $f$, as stated next. The proofs work, however, for any smooth stable solution
to problem \eqref{problem} under the same assumptions on $f$. From now on
in the paper, we adopt this framework and describe
a priori estimates for smooth stable solutions.
In this direction, an interesting question is
whether a stable weak solution of \eqref{problem} can be approximated by smooth stable 
solutions to similar problems ---to which we could apply the available a priori estimates and 
in this way deduce
regularity of the given weak solution. A nice answer was given in \cite{BCMR} (see also Corollary 3.2.1 of 
the monograph~\cite{Dupaigne})
and states that this is always possible if $u\in H^1_0(\Omega)$ is a stable weak solution of \eqref{problem}
and $f$ is nonnegative and convex. The approximating functions can be taken to be smooth stable solutions
to problem \eqref{problem} with $f$ replaced by $(1-\varepsilon)f$, where $\varepsilon\downarrow 0$.

In 1975, Crandall and Rabinowitz \cite{CR} proved an $L^{\infty}(\Omega)$ bound for stable solutions whenever 
$n\leq 9$ and $f(u)=\lambda e^u$ ---as well as for $f(u)=\lambda (1+u)^m$, $m>1$.
This was improved by Sanch\'on~\cite{Sanchon}, who established that $u\in L^\infty(\Omega)$ whenever
$n\leq 9$ and $f\in C^{2}$ satisfies \eqref{eq1.30} as well as that
\begin{equation}\label{tau}
\text{ the limit } \lim_{t \to +\infty} \frac{f(t)f''(t)}{f'(t)^2} \text{ exists.}
\end{equation}

On the other hand, in the radial case $\Omega = B_{1}$
Capella and the author~\cite{CC} proved an $L^\infty$ bound when $n\leq 9$ for every $C^1$ nonlinearity $f$. 

The work of Nedev \cite{Nedev} established the $L^\infty(\Omega)$ bound for
$n=2$ and $3$ when $f$ satisfies \eqref{eq1.30}.
When $2\leq n \leq 4$, in 2010 the author~\cite{Cabre} established that the
$L^{\infty}$ bound holds for every smooth $f$ if in addition $\Omega$ is convex.
Villegas \cite{V} extended this bound for $n=4$ to nonconvex domains
if $f$ is assumed to satisfy \eqref{eq1.30} ---in particular $f$ must be convex. 
He used both the results of \cite{Cabre}
and \cite{Nedev}.

The proof in \cite{Cabre} relied on a quite delicate application of 
the Michael-Simon and Allard Sobolev inequality on every level set of the solution. In addition, 
the proof is very different from that of the radial case up to dimension 9. 
These two points will be overcome with the new proof in the current paper, which we start
describing in next subsection.

When $5\leq n\leq 9$, 
the existence of an $L^\infty$ bound holding for
all smooth nonlinearities was an open question since the mid nineties. It has been very recently
solved by Figalli, Ros-Oton, Serra, and the author~\cite{CFRS}, for nonnegative nonlinearities, 
using different ideas from those of the current article.

Let us address now an important issue: the choice of test function $\xi$ to be used in the stability condition
\eqref{stability} to deduce the estimates. In the works of Crandall-Rabinowitz, Sanch\'on, and Nedev
the test function was 
$$
\xi = h(u)
$$
for a certain, well chosen, function $h$ that depends only on the nonlinearity $f$. 
Instead, in the radial result by Capella and the author, we
chose
\begin{equation*}
\xi=u_r r^{-\beta} \zeta
\end{equation*}
for an appropriate exponent $\beta>0$. Here $r=|x|$ and $\zeta$ is a function identically one except in small a
neighborhood of $\partial\Omega$, where it vanishes ---and also vanishing in a small ball around $\{x=0\}$ 
in order to regularize the singularity of $r^{-\beta}$, though this has no implication in the final computations. 
For the estimate in \cite{Cabre} up to dimension~4 in the nonradial case we used
\begin{equation*}
\xi=\left|\nabla u\right|\varphi(u),
\end{equation*}
with $\varphi$ chosen in a quite subtle way depending on the solution $u$ itself.
In the proof of the current paper we will use instead
\begin{equation*}
\xi(x)=\left|\nabla u(x)\right| |x-y|^{-\beta} \zeta (x),
\end{equation*}
with $y\in\Omega$, $\beta>0$, and $\zeta$ a cut-off function vanishing on $\partial\Omega$.
With $y=0$, in the radial case this test function becomes the radial test function above.
The forthcoming paper by Figalli, Ros-Oton, Serra, and the author~\cite{CFRS} will use still
a different test function to get a new key estimate.

\subsection{New results and description of their proof}

The following is the main result of the article. It establishes an $L^\infty$ a priori estimate for
all nonlinearities and domains in dimensions 3 and 4. After its statement, 
we will comment on the analogous result in dimension~2.

\begin{theorem}\label{Thm1}
Let $f$ be any $C^1$ nonlinearity, $\Omega\subset \mathbb{R}^n$ a smooth bounded domain,
and $u$ a smooth stable solution of equation \eqref{eqsem}.
Assume that $3\leq n\leq 4$.

Then, given any $\delta>0$, we have that
\begin{equation}
\label{est4}
 \Vert u\Vert_{L^\infty(K_\delta)} \leq C \left( \Vert u\Vert_{L^1(\Omega)}
 +\Vert \nabla u\Vert_{L^2(\Omega\setminus K_\delta)} \right),
\end{equation}
where
\begin{equation*}
K_\delta:=\{x\in\Omega : \dist(x,\partial\Omega)\geq\delta\}
\end{equation*}
and $C$ is a constant depending only on $\Omega$ and $\delta$.
In particular, $C$ is independent of~$f$.

If in addition $\Omega$ is convex, $u$ solves problem \eqref{problem}, and $u>0$ in $\Omega$, then
\begin{equation}\label{est5}
 \Vert u\Vert_{L^\infty(\Omega)} \leq C
\end{equation}
for some constant $C$ depending only on $\Omega$, $f$, and $\Vert u\Vert_{L^1(\Omega)}$.

If $\Omega$ is a ball, the same bounds \eqref{est4}-\eqref{est5} 
hold if $2\leq n\leq 9$. Here, the assumption $u>0$ in $\Omega$ is not needed to ensure \eqref{est5}.
\end{theorem}

Estimate \eqref{est5}, which will follow from \eqref{est4}, was first established in \cite{Cabre} by the author.
In that paper, however, it followed from the estimate
\begin{equation}\label{dim4}
 \Vert u\Vert_{L^\infty(\Omega)} \leq t + \frac{C}{t}|\Omega|^{(4-n)/(2n)}
\left( \int_{\{u <t\}} |\nabla u|^4 \ dx\right)^{1/2} \quad \text{ for every }t>0
\end{equation}
---instead of from \eqref{est4}. In \eqref{dim4}, $C$ is a universal constant, in particular independent of 
$f$ and $\Omega$. 

For the application to the global estimate \eqref{est5} in convex domains, it is essential that the $L^2$
and $L^4$ norms of $\nabla u$ in \eqref{est4} and \eqref{dim4}, respectively, are computed in
a small neighborhood of $\partial\Omega$ ---which is the case if $\delta$ and $t$, respectively, 
are small enough.

In the nonradial case, \eqref{est4} is the first interior $L^\infty$ estimate for stable solutions of
equation \eqref{eqsem} which holds independently of their boundary values.

In \cite{Cabre}, \eqref{dim4} and \eqref{est5} were proven also for $n=2$, as a rather simple 
consequence of Proposition~\ref{geomstab} below and the Gauss-Bonnet formula. 
Instead, the proof in the current paper only gives an estimate for $n=2$ in the radial case.
However, it is immediate to see that if $u=u(x_1,x_2)$ is a stable solution in a domain $\Omega$ of $\R^2$, 
then the function $v(x_1,x_2,x_3):=u(x_1,x_2)$ is also a stable solution of the same equation in any
domain of $\R^3$ contained in $\Omega\times\R$. As a consequence, from estimate \eqref{est4} in dimension 3, 
we deduce the same bound \eqref{est4} in dimension~2 but with its right hand side replaced by 
$C \Vert u\Vert_{H^1(\Omega)}$.

Theorem \ref{Thm1} will follow easily from the following key estimate.
It is a bound on a weighted Dirichlet integral
for stable solutions in every dimension.

\begin{proposition}
 \label{weightest}
 Let $f$ be any $C^1$ nonlinearity, $\Omega\subset \mathbb{R}^n$ a smooth bounded domain,
and $u$ a smooth stable solution of equation \eqref{eqsem}.

Let $\alpha$ satisfy
\begin{equation}
\label{condalpha}
0\leq \alpha < n-1\quad\text{ and }\quad (\alpha -2)^2/4 < (n-1-\alpha)^2/(n-1).
\end{equation}
Then, for all $\delta>0$ and $y\in K_\delta:=\{x\in\Omega : \dist(x,\partial\Omega)\geq\delta\}$, we have that
\begin{equation}\label{estgrad}
 \int_{\Omega}  |\nabla u(x)|^2 |x-y|^{-\alpha} \,dx  \leq C \Vert \nabla u\Vert^2_{L^2(\Omega\setminus K_\delta)}
\end{equation}
for some constant $C$ depending only on $\Omega$, $\delta$, and $\alpha$.

Finally, if $\Omega$ is a ball then \eqref{estgrad} holds with $y=0$ if instead of 
\eqref{condalpha} we assume
\begin{equation}
\label{condalpharad}
0\leq \alpha < n-1\quad\text{ and }\quad (\alpha -2)^2/4 < n-1.
\end{equation}
\end{proposition}

This result will easily give, in dimension 5, an {\em interior} $L^p$ bound for all finite exponent $p$.
To have an $L^p$ bound up to the boundary we need to assume $\Omega$ to be convex, as stated in the next result.

\begin{corollary}\label{Corol5}
Assume that $n=5$. Let $f$ be any $C^1$ nonlinearity, $\Omega\subset \mathbb{R}^5$ a convex smooth bounded domain,
and $u>0$ a smooth stable solution of problem \eqref{problem}.

Then,
\begin{equation}
\label{estcor5}
 \Vert u\Vert_{L^p(\Omega)} \leq C_p \quad\text{ for every } 1<p<\infty,
 \end{equation}
where $C_p$ is a constant depending only on $p$, $\Omega$, $f$, and $\Vert u\Vert_{L^1(\Omega)}$.
\end{corollary}

Our key estimate in Proposition \ref{weightest} will follow from two tools:
\begin{itemize}
 \item the bound \eqref{semi1} for stable solutions of Sternberg-Zumbrun, explained below, and
 \item a new geometric Hardy inequality on hypersurfaces of $\R^n$, \eqref{hardy_ibp_p2} below, having universal
 constants ---as in the Michael-Simon and Allard Sobolev inequality. 
 Its proof will be simpler than the one of the Michael-Simon and Allard Sobolev inequality. Once it
 is applied to every level set of an arbitrary function~$u$, it becomes Theorem~\ref{hardyfolprop} below.
\end{itemize}

To describe these two tools, we need to introduce some geometric objects.
First recall that since $u$ is smooth, by Sard's theorem, almost every $t\in\R$ is a regular value of 
$u$. By definition, if $t$ is a regular value of $u$, then 
$\left|\nabla u(x)\right|>0$ for all $x\in\Omega$ such that 
$u(x)=t$. In particular, if $t$ is a regular value, $\{u=t\}$ is (if not empty)
a smooth embedded hypersurface of $\mathbb{R}^n$.

In the set of regular points of $u$, $\left\{x\in\Omega :  |\nabla u(x)|>0\right\}$, we consider the 
normal vector to the level sets of $u$,
\begin{equation}\label{gradnormal}
\nu= \nabla u / |\nabla u|,
\end{equation}
as well as the tangential gradient along the level sets. That is,
for a $C^1(\Omega)$ function $\varphi$ we consider 
\begin{equation}\label{gradtang}
\nabla_T \varphi := \nabla \varphi - \left \langle \nabla \varphi, \nu \right \rangle \nu
\end{equation}
---the projection of the full 
gradient of $\varphi$ at $x$ onto the tangent space to the level set of $u$ passing through $x$.
We will also encounter the square of the second fundamental form of the level sets,
$$
|A|^2=|A(x)|^2=\sum_{l=1}^{n-1} \kappa_l^2,
$$
where $\kappa_l$ are the principal curvatures of the level set of $u$ passing
through $x$. In other results, it will be of importance the mean curvature of the level
sets, defined as
\begin{equation}\label{defH}
H=H(x)=\sum_{l=1}^{n-1} \kappa_l.
\end{equation}

The level sets of a solution $u$, and their curvatures, appeared in the following 
important result of Sternberg and Zumbrun \cite{SZ1,SZ2}. It is an inequality 
that follows from the stability hypothesis \eqref{stability} together with the identity
\begin{equation}\label{eq:grad}
\left(\Delta+f'(u)\right)\left|\nabla u\right|=
\frac{1}{\left|\nabla u\right|} 
\left(\left|\nabla_T\left|\nabla u\right|\right|^2+
\left|A\right|^2\left|\nabla u\right|^2\right)\ \text{in}\ 
\Omega\cap\left\{\left|\nabla u\right|>0\right\}.
\end{equation}
By taking $\xi=\left|\nabla u\right|\eta$ in the stability hypothesis \eqref{stability}, the presence of 
$f'(u)$ in \eqref{stability} disappears, since the left hand 
side of \eqref{eq:grad} refers to $\Delta+f'(u)$. One then obtains the following bound for stable solutions. 
See also \cite{Cabre} for a proof of the proposition.

\begin{proposition}[Sternberg-Zumbrun \cite{SZ1,SZ2}]\label{geomstab} 
Let $f$ be any $C^1$ nonlinearity, $\qquad\Omega\subset \mathbb{R}^n$ a smooth bounded domain,
and $u$ a smooth stable solution of equation \eqref{eqsem}.

Then, 
\begin{equation}\label{semi1}
\int_{\Omega\cap\{|\nabla u|>0\}} 
\left( |\nabla_T |\nabla u||^2 +|A|^2|\nabla u|^2\right)\eta^2 \,dx
\leq \int_\Omega |\nabla u|^2 |\nabla \eta|^2 \,dx
\end{equation}
for every $C^1(\overline\Omega)$ function 
$\eta$ with $\eta|_{\partial\Omega}\equiv0$.
\end{proposition}

In \cite{Cabre}, this result was used choosing $\eta= \varphi (u)$, applying then the coarea formula,
and estimating by below the left hand side of \eqref{semi1} through the Michael-Simon and Allard Sobolev 
inequality applied on every level set of $u$. Recall that this is a remarkable
Sobolev inequality for functions defined in general hypersurfaces of~$\R^n$. The constants
appearing in the inequality are universal. This is accomplished by adding to the usual
Dirichlet norm on the hypersurface, an additional $L^2$ norm weighted with the mean curvature
of the hypersurface. 

Instead, in this paper we bound by below the left hand side of \eqref{semi1} through a simpler tool: 
a new Hardy inequality on hypersurfaces, also with universal constants and an additional term 
involving the mean curvature. 
In fact, what we need is the following geometric Hardy inequality on the foliation of hypersurfaces given by the
level sets of a function $u$. To state it, given $y\in\R^n$ we introduce the functions (of~$x$)
$$
r_y= r_y(x)= |x-y|
$$
and
$$
u_{r_y}=u_{r_y}(x)=\nabla u(x)\cdot \frac{x-y}{|x-y|}.
$$

The following is the new geometric Hardy inequality. Recall \eqref{gradnormal}-\eqref{defH}
for the meaning of $\nabla_T$ and $H$ (which refer to the level sets of $u$) in the statement. 
As we will explain later, our inequalities have the best constants in the flat case, i.e., 
when the level sets of $u$ are hyperplanes.

\begin{theorem}
 \label{hardyfolprop}
Let $\Omega\subset \mathbb{R}^n$ be a smooth bounded domain, 
$u$ any $C^\infty(\overline\Omega)$ function, $\varphi\in C^1(\overline\Omega)$, 
$\alpha\in [0,n-1)$, and $y\in\R^n$. Assume that either $u|_{\partial\Omega}\equiv 0$ or
$\varphi|_{\partial\Omega}\equiv 0$.

Then,  we have
 \begin{equation}
\begin{split}
  \label{hardyfol}
 &\hspace{-3mm}(n-1-\alpha) \int_\Omega |\nabla u| \varphi^2 r_y^{-\alpha} \,dx
 + \alpha \int_{\Omega} \frac{u_{r_y}^2}{|\nabla u|} \varphi^2 r_y^{-\alpha} \,dx\\
  &\hspace{3mm}\leq   \left( \int_\Omega |\nabla u| \varphi^2 r_y^{-\alpha}\,dx \right)^{1/2}
  \left( \int_{\Omega\cap\{|\nabla u|>0\}} |\nabla u| \left( 4|\nabla_T\varphi|^2 + H^2\varphi^2\right) r_y^{2-\alpha} \,dx \right)^{1/2}.
 \end{split}
 \end{equation}
 In particular,
 \begin{equation}
  \label{hardymanif}
(n-1-\alpha)^2 \int_\Omega |\nabla u| \varphi^2 r_y^{-\alpha} \,dx
 \leq  \int_{\Omega\cap\{|\nabla u|>0\}} |\nabla u| \left( 4|\nabla_T\varphi|^2 + H^2 \varphi^2\right) r_y^{2-\alpha}\,dx .
 \end{equation}
 
In addition, if $u=u(r_y)$ is radially symmetric about the point $y$, then 
\begin{equation}\label{hardyrad}
(n-1)^2 \int_{\Omega}  |u_{r_y}| \varphi^2 r_y^{-\alpha} \,dx 
\leq \int_{\Omega\cap\{|\nabla u|>0\}}  |u_{r_y}|   \big (4 |\nabla_T \varphi |^2 + {H}^2\varphi^2 \big )r_y^{2-\alpha}  \,dx.
\end{equation}
\end{theorem}

To deduce our results on stable solutions, we will apply these inequalities with
$$
\varphi = |\nabla u|^{1/2}\zeta
$$
---appropriately regularized at critical points of $u$ and with $\zeta$ being a cut-off.
The values of the constants in the inequalities of Theorem \ref{hardyfolprop}
will determine the dimensions in which we can 
prove boundedness of stable solutions. In particular, the larger constant in the left hand side of 
\eqref{hardyrad} ---with respect to that in \eqref{hardymanif}--- will allow to reach the optimal dimension 9
in the radial case.

Note that, as in the Michael-Simon and Allard inequality, 
the constants in the previous inequalities are universal and the geometry of the level sets of $u$
only appears through the term involving their mean curvature. In fact, the level sets of $u$ are what
matters in the inequalities, and not $u$ itself. More precisely, given a hypersurface $M$ of $\R^n$, consider 
its parallel hypersurfaces (in a $\varepsilon$-neighborhood of $M$) and the function $u=u_\varepsilon$ to be 
$\varepsilon^{-1}\dist(\cdot,M)$ in the $\varepsilon$-neighborhood of $M$. Letting $\varepsilon\to 0$
in the inequalities \eqref{hardyfol} and \eqref{hardymanif} corresponding to $u=u_\varepsilon$
(and $\alpha=2$ and $y=0$ to simplify), one obtains (thanks to the coarea formula)
the following Hardy inequality for functions defined on a single hypersurface $M\subset \R^n$. For every
$C^1$ function $\varphi$ with compact support in $M$, a hypersurface of $\R^n$, we have
\begin{multline}
\label{hardy_ibp_p2}
\quad\qquad ((n-1)-2)\int_M\frac{\varphi^2}{|x|^2}\,dV+2\int_M\left(\frac{x}{|x|}\cdot\nu\right)^2
\frac{\varphi^2}{|x|^2} \, dV
\\ \leq \left(\int_M\frac{\varphi^2}{|x|^2}\,dV\right)^{1/2}
\left(\int_M\left(4|\nabla_T\varphi|^2+H^2\varphi^2\right)\,dV\right)^{1/2},
\end{multline}
where $H$ is the mean curvature of $M$ and $\nu$ the normal to $M$. In particular,
\begin{equation}
\label{hardy_ibp}
\frac{((n-1)-2)^2}{4}\int_M\frac{\varphi^2}{|x|^2}\,dV\leq 
\int_M\left(|\nabla_T\varphi|^2+\frac{1}{4} H^2\varphi^2\right)\,dV.
\end{equation}

And vice versa, from the inequalities \eqref{hardy_ibp_p2} and \eqref{hardy_ibp} in a (general) 
single hypersurface $M$, we can deduce Theorem~\ref{hardyfolprop}.
For this, we just write the integrals in \eqref{hardyfol} and \eqref{hardymanif} through the coarea formula
and recall Sard's theorem as described above. 
For instance, we have 
$$
\int_{\Omega}  |\nabla u|\frac{\varphi^2}{|x|^2} \,dx =\int_\R \left( \int_{\Omega\cap\{u=t\}} 
\frac{\varphi^2}{|x|^2}\, dV\right) dt.
$$

Carron \cite{Car} had already established a universal Hardy inequality in hypersurfaces of $\R^{n}$.
However, it differs from ours, (\ref{hardy_ibp}), in the mean curvature term. His estimate states that
if $n-1\geq 3 $ and $\varphi$ is a $C^1$ function with compact support in~$M$, then
\begin{equation}\label{hardy_car}
\frac{((n-1)-2)^2}{4}\int_M\frac{\varphi^2}{|x|^2}\,dV \leq 
\int_M \left(|\nabla_T\varphi|^2+ \frac{(n-1)-2}{2}\, \frac{|H|}{|x|}\, \varphi^2 \right)\,dV.
\end{equation}
Note that this inequality and Cauchy-Schwarz lead to one like ours, of the form \eqref{hardy_ibp},
but with worst constants ---and recall that the value of constants are crucial in our applications.
In addition, our inequality \eqref{hardy_ibp_p2} with an extra term in the left hand side will
be needed to reach dimension 9 in the radial case.

Note also that, if $ M=\R^{n-1}$, both \eqref{hardy_ibp} and \eqref{hardy_car} become the classical
Hardy inequality in $\R^{n-1}$ with its best constant.

Further geometric Hardy-Sobolev inequalities in hypersurfaces of $\R^n$ will be studied in
collaboration with P. Miraglio in~\cite{CabMir}.

\section{The geometric Hardy inequality}

To prove Theorem \ref{hardyfolprop} we will use the tangential derivatives to the level sets of $u$, defined by
\begin{equation*}
\label{deltas}
\delta_i \varphi := \partial_i \varphi - (\partial_k \varphi) \nu^k\nu^i
\end{equation*}
for $i=1,\ldots,n$. Here, as in the rest of the paper, we used the standard summation convention
over repeated indices. Recalling the definition \eqref{gradtang} of the tangential gradient $\nabla_T$, 
it is easy to verify that
\begin{equation}
\label{tanggrad}
\sum_{i=1}^n (\delta_i \varphi)^2 =|\nabla_T\varphi|^2.
\end{equation}
Using that $\nu^i\partial_k\nu^i=0$ (a consequence of $|\nu|^2\equiv 1$), one can verify the following
identities defining the mean curvature of the level sets of $u$:
\begin{equation*}
 \label{meancurv}
 H:=\sum_{i=1}^n \delta_i\nu^i=\sum_{i=1}^n \partial_i\nu^i= \diverg \nu.
\end{equation*}

We are now ready to state and prove the formula of integration by parts on the whole family of level sets of $u$.
Since $\nu^{i}$ and $\delta_i$ are only defined where $|\nabla u|$ does not vanish, 
the next integrals are computed on this set.

\begin{lemma}
 \label{intpartslem}
Let $\Omega\subset \mathbb{R}^n$ be a smooth bounded domain and  
$u$ be any $C^\infty(\overline\Omega)$ function. Let $\varphi$ and $\psi$ belong to $C^1(\overline\Omega)$. 
Assume that either $u|_{\partial\Omega}\equiv 0$ or $\varphi|_{\partial\Omega}\equiv 0$.

Then, for every $i\in\{1,\ldots,n\}$,  we have
 \begin{equation}
  \label{intparts}
  \int_{\Omega\cap\{|\nabla u|>0\}} |\nabla u| (\delta_i\varphi)\psi \,dx=
  -\int_{\Omega\cap\{|\nabla u|>0\}} |\nabla u| \varphi\, \delta_i\psi \,dx
  +\int_{\Omega\cap\{|\nabla u|>0\}} |\nabla u| H\nu^i\varphi\psi \,dx.
 \end{equation}
\end{lemma}

Through the coarea formula, \eqref{intparts} follows directly from the well known formula of integration
by parts on hypersurfaces of Euclidean space. However, for completeness, we next give a simple proof of it
where we take advantage of having the foliation by level sets.

\begin{proof}[Proof of Lemma \ref{intpartslem}]
It suffices to prove the identity with $\psi\equiv 1$, that is,
\begin{equation}
\label{hardy1}
 \int_{\Omega\cap\{|\nabla u|>0\}} |\nabla u| \delta_i\varphi \,dx   =  \int_{\Omega\cap\{|\nabla u|>0\}} 
 |\nabla u| H\nu^i\varphi\,dx   .
\end{equation}
Indeed, by replacing here $\varphi$ by $\varphi\psi$, one concludes \eqref{intparts}. Through the rest of the paper,
sometimes we will use the notation
$$
\phi_i := \partial_i \phi \qquad\text{and}\qquad \phi_{ij} := \partial_{ij} \phi
$$
for first and second order partial derivatives.

Since we will integrate by parts in $R:=\{x\in\Omega : |\nabla u(x)|>0\}$ (a set that might not be smooth),  
we consider the open sets
$$
R_{\varepsilon}:= \{x\in\Omega : |\nabla u(x)|>\varepsilon\}
$$
for $\varepsilon >0$.
Since $|\nabla u|$ is a smooth function in $R$, Sard's theorem ensures  
that $\Omega\cap\partial R_{\varepsilon}=R\cap\partial R_{\varepsilon}$ is a smooth 
hypersurface for almost all $\varepsilon$.  
Denote by $\nu_{R_{\varepsilon}}$ the exterior unit normal to $R_{\varepsilon}$ 
---not to be confused with the normal $\nu=\nabla u/|\nabla u|$ to the level sets of $u$.
Since part of $\partial R_\varepsilon$ could be contained in $\partial\Omega$, and another part in~$\Omega$,
the open set $R_{\varepsilon}$ is piecewise smooth.
Hence, we can integrate by parts in $R_{\varepsilon}$ to get

\begin{align}\label{intp1}
\int_{R_{\varepsilon}}  |\nabla u|\delta_i \varphi \,dx   
&=  \int_{R_{\varepsilon}}  |\nabla u|  \left( \varphi_i - \varphi_k \nu^k \nu^i
\right) \,dx \\
\label{intp2}&= - \int_{R_{\varepsilon}}  \left(\partial_i |\nabla u| -  \partial_k (|\nabla u| \nu^k \nu^i) 
\right) \varphi 
\,dx \\
\label{intp3}&\hspace{1cm}+\int_{\partial R_{\varepsilon}}  |\nabla u| \varphi \left( \nu_{R_{\varepsilon}}^i -
\nu^k\nu^i\nu_{R_{\varepsilon}}^k \right)\,dV.
\end{align}

Note that the first integral in \eqref{intp1} tends to the left hand side of \eqref{hardy1} as  
$\varepsilon\to 0$.


Next, we deal with \eqref{intp2}. Since   
\begin{equation*}
\label{Eq:PartialModulusOfGradv}
\partial_j |\nabla u| = \dfrac{u_{k j} u_k}{|\nabla u|} =u_{k j} \nu^k,
\end{equation*}
we deduce
\begin{align*}
\label{Eq:ProofIntegrationByPartsEq1}
 \int_{R_{\varepsilon}}  \left(\partial_i |\nabla u| -  \partial_k (|\nabla u| \nu^k \nu^i) \right) \varphi 
\,dx  &\\
&\hspace{-4.5cm}=  \int_{R_{\varepsilon}}   \left( u_{k i} \nu^k  -  u_{j k} \nu^j \nu^k \nu^i  - |\nabla u| \partial_k ( \nu^k \nu^i)  \right) \varphi \,dx \\
&\hspace{-4.5cm}=  \int_{R_{\varepsilon}}   \left( u_{k i} \nu^k  -  u_{j k} \nu^j \nu^k \nu^i  - |\nabla u| \delta_k ( \nu^k \nu^i)    - |\nabla u| \partial_j ( \nu^k \nu^i) \nu^j\nu^k   \right) \varphi \,dx.
\end{align*}
From this, using that $\nu^k\delta_k\psi\equiv 0$ for all functions $\psi$, that $\nu^k\partial_j\nu^k\equiv 0$, 
and that 
\begin{equation}\label{bddnormal}
\partial_j \nu^i = \partial_j \left( \dfrac{u_i}{|\nabla u|} \right) =  \dfrac{u_{i j}}{|\nabla u|} - 
\dfrac{u_{kj} u_{k} u_i}{|\nabla u|^3},
\end{equation}
we conclude
\begin{align*}
 \int_{R_{\varepsilon}}  \left(\partial_i |\nabla u| -  \partial_k (|\nabla u| \nu^k \nu^i) \right) \varphi 
\,dx  &\\
&\hspace{-4.5cm}
= \int_{R_{\varepsilon}}   \left( u_{k i} \nu^k  -  u_{j k} \nu^j \nu^k \nu^i  - |\nabla u| H \nu^i  
- \nu^j ( u_{ij}-u_{kj}\nu^k \nu^i)  \right) \varphi \,dx\\
&\hspace{-4.5cm}=- \int_{R_{\varepsilon}}   |\nabla u| H \nu^i   \varphi \,dx. 
\end{align*}
Since by \eqref{bddnormal} we have that $|\nabla u| |H| \leq C|D^2 u|\leq C$ for some constants $C$,
the last integral tends, up to its sign and as  
$\varepsilon\to 0$, to the right hand side of \eqref{hardy1}.

Finally, we deal with the boundary term \eqref{intp3}, which is the most delicate. 
Recall that we are assuming that either $u|_{\partial\Omega}\equiv 0$ or $\varphi|_{\partial\Omega}\equiv 0$, 
and that $\nu$ is a normal to the level sets of $u$. Thus, in the first case $u|_{\partial\Omega}\equiv 0$, 
we will have $\nu=\pm\nu_{R_{\varepsilon}}$ at every point on $\partial R_{\varepsilon}\cap\partial \Omega$. Hence
$\int_{\partial R_{\varepsilon}\cap \partial \Omega}  |\nabla u| \varphi ( \nu_{R_{\varepsilon}}^i -
\nu^k\nu^i\nu_{R_{\varepsilon}}^k )\,dV=0$. In the second case, $\varphi|_{\partial\Omega}\equiv 0$, 
this last integral will vanish. Therefore, \eqref{intp3} is equal to
\begin{equation}\label{avedelta}
\int_{\Omega \cap\partial R_{\varepsilon}}  |\nabla u| \varphi ( \nu_{R_{\varepsilon}}^i -
\nu^k\nu^i\nu_{R_{\varepsilon}}^k )\,dV =
\int_{\Omega \cap\{ |\nabla u|=\varepsilon\}}  |\nabla u| \varphi ( \nu_{R_{\varepsilon}}^i -
\nu^k\nu^i\nu_{R_{\varepsilon}}^k )\,dV.
\end{equation}
Even that the absolute value of the integrand is bounded by a constant times $\varepsilon$, we can not
conclude that \eqref{avedelta} tends to zero as $\varepsilon\to 0$ ---since we do not have control on the 
surface measure of $\{ |\nabla u|=\varepsilon\}$. To remedy this, given $\delta >0$, we average 
\eqref{intp1}-\eqref{intp3} in $\varepsilon\in (0,\delta)$. We already know that the terms
\eqref{intp1} and \eqref{intp2} will converge, as $\delta\downarrow 0$, to the desired quantities in \eqref{hardy1}.
Therefore, it suffices to check that 
\begin{equation*}
\frac{1}{\delta}\int_0^\delta d\varepsilon \int_{\Omega \cap\{ |\nabla u|=\varepsilon\}} 
dV\,  |\nabla u| \varphi ( \nu_{R_{\varepsilon}}^i -
\nu^k\nu^i\nu_{R_{\varepsilon}}^k )
\end{equation*}
tends to zero as $\delta\downarrow 0$. But using the coarea formula, the absolute value of this 
quantity can be bounded by
\begin{align*}
\frac{C}{\delta}\int_0^\delta d\varepsilon \int_{\Omega \cap\{ |\nabla u|=\varepsilon\}} dV\,  |\nabla u| 
& 
\\
&\hspace{-3.5cm}
\leq \frac{C}{\delta}\int_{\{ 0<|\nabla u|<\delta \}}    \left|\nabla |\nabla u|\right|  |\nabla u|\, dx
= \frac{C}{\delta}\int_{\{ 0<|\nabla u|<\delta \}} \left|\nabla (|\nabla u|^2/2)\right| \, dx
\\
&\hspace{-3.5cm}
\leq \frac{C}{\delta}\int_{\{ 0<|\nabla u|<\delta \}} |D^2 u||\nabla u| \, dx
\leq \frac{C}{\delta}\int_{\{ 0<|\nabla u|<\delta \}} |\nabla u| \, dx
\\
&\hspace{-3.5cm}
\leq C \left| \{ 0<|\nabla u|<\delta \} \right|
\end{align*}
for different constants $C$. The last quantity tends to zero as $\delta\downarrow 0$, and thus 
the proof is finished. 
\end{proof}

We can now establish the geometric Hardy inequality.

\begin{proof}[Proof of Theorem \ref{hardyfolprop}]
First, note that (denoting by $\delta_{ik}$ the Kronecker delta)
\begin{equation}
\label{Eq:HadyTypeIneqProof1}
    \delta_i x_i:= \sum_{i=1}^n \delta_i x_i =  \sum_{i=1}^n (1-\delta_{ik} \nu^k\nu^i) = n - 1.
\end{equation}
We next use this fact and the geometric integration by parts formula of Lemma~\ref{intpartslem} to
compute $\int_{\Omega}  |\nabla u| \varphi^2 r_y^{-\alpha} \,dx$. We may assume, after a translation,
that $y=0$; we will denote then $r_y=r=|x|$. Recall that 
$\alpha\in [0,n-1)$, $\varphi\in C^1(\overline\Omega)$, and 
either $u|_{\partial\Omega}\equiv 0$ or $\varphi|_{\partial\Omega}\equiv 0$.
The following computations must be done with $r^{-\alpha}$ replaced by a nice regularization $\zeta_\varepsilon$
of it in an $\varepsilon$-neighborhood of $y=0$. Since all terms in the following
computations are given by integrable functions, through dominated convergence we can let
$\varepsilon\to 0$ to justify the following equalities, which we write directly for $r^{-\alpha}$
instead of $\zeta_\varepsilon$. 

Since
$$
x_i \delta_i r^{-\alpha}=\sum_{i=1}^n x_i \delta_i r^{-\alpha} =  -\alpha r^{-\alpha} 
\left( 1 - \left ( x\cdot \nu/r\right )^2\right),
$$
we have
\begin{align*}
(n-1) \int_{\Omega\cap\{|\nabla u|>0\}}  |\nabla u| \varphi^2 r^{-\alpha} \,dx &=
 \int_{\Omega\cap\{|\nabla u|>0\}}  |\nabla u| \varphi^2 r^{-\alpha} \delta_i x_i \,dx \\
& \hspace{-5.5cm} = - \int_{\Omega\cap\{|\nabla u|>0\}}  
|\nabla u| \left( 2 \varphi (\delta_i \varphi) r^{-\alpha}  x_i 
+ \varphi^2 x_i  \delta_i r^{-\alpha} - H \varphi^2 r^{-\alpha} x_i \nu^i \right) \,dx \\
&\hspace{-5.5cm} =- \int_{\Omega\cap\{|\nabla u|>0\}}  |\nabla u| \left( 2 \varphi (\delta_i \varphi) r^{-\alpha}  
x_i   -\alpha \varphi^2 r^{-\alpha} \left( 1 - \left (  x\cdot \nu/r\right )^2\right ) -
 H \varphi^2 r^{-\alpha} x_i \nu^i \right) \,dx .
\end{align*}
Therefore,
\begin{align*}
(n-1) \int_{\Omega\cap\{|\nabla u|>0\}}  |\nabla u| \varphi^2 r^{-\alpha} \,dx 
- \alpha  \int_{\Omega\cap\{|\nabla u|>0\}}  |\nabla u| \varphi^2 r^{-\alpha} 
\left(  1 - \left (  x\cdot \nu/r\right )^2\right) \,dx  \\
&\hspace{-12cm} =  - \int_{\Omega\cap\{|\nabla u|>0\}}  |\nabla u| \varphi r^{1-\alpha}  \dfrac{x_i}{r} \left( 2 \delta_i \varphi  -
 H \varphi \nu^i\right) \,dx.
 \end{align*}
 
 Finally, note that
 \begin{align*}
 \sum_{i=1}^n (2\delta_i\varphi -H\varphi  \nu^i)^2 &= 
 \sum_{i=1}^n \left( 4(\delta_i\varphi)^2 -4H(\delta_i\varphi)\varphi  \nu^i +H^2\varphi^2(\nu^i)^2\right)\\
 &=\sum_{i=1}^n \left( 4(\delta_i\varphi)^2 +H^2\varphi^2(\nu^i)^2\right) \\
 &= 4 |\nabla_T \varphi |^2 + {H}^2\varphi^2,
 \end{align*}
 where we have used \eqref{tanggrad}. We now conclude, using the Cauchy-Schwarz inequality,
\begin{align*}
(n-1-\alpha) \int_{\Omega\cap\{|\nabla u|>0\}}  |\nabla u| \varphi^2 r^{-\alpha} \,dx 
+ \alpha  \int_{\Omega\cap\{|\nabla u|>0\}}  \frac{u_{r}^2}{|\nabla u|} \varphi^2 r^{-\alpha}  \,dx  &\\
&\hspace{-11.8cm} \leq  \left ( \int_{\Omega\cap\{|\nabla u|>0\}} 
|\nabla u|  \varphi^2 r^{-\alpha} \,dx\right )^{1/2} \left ( \int_{\Omega\cap\{|\nabla u|>0\}}  
|\nabla u|  \big (4 |\nabla_T \varphi |^2 + {H}^2\varphi^2 \big )  r^{2-\alpha} \,dx\right )^{1/2}.
\end{align*}
Since $u_{r}^2/|\nabla u|\leq |\nabla u|$ vanishes in
the critical set $\{|\nabla u|=0\}$, we can replace by the whole $\Omega$ the sets of integration in the 
left hand side of the inequality. This establishes \eqref{hardyfol}. 
From it, the last two statements of the theorem follow easily.
\end{proof}

\section{Proof of the $L^\infty$ bounds} 

Here we use the geometric Hardy inequality together with the Sternberg-Zumbrun stability condition
to prove our main estimates.

\begin{proof}[Proof of Proposition \ref{weightest}]
Given $\delta >0$, let $K_\delta =\{ x\in\Omega : \dist(x,\partial\Omega)\geq\delta\}$.
For $\alpha\in [0,n-1)$, we apply the geometric Hardy inequality \eqref{hardyfol} with
 $$
\varphi := |\nabla u|^{1/2} \zeta,
$$
where 
\begin{equation}
 \label{propzeta1}
 \zeta_{|\partial\Omega}\equiv 0\qquad\text{ and }\qquad \zeta\equiv 1 \text{ in } K_{\delta/2}.
\end{equation}
To be totally rigorous, the proof should be carried out with $\varphi$ replaced by
$$
\varphi_{\varepsilon} := (|\nabla u|^{2} +\kappa^2)^{1/4}\zeta,
$$
and letting $\kappa$ tend to zero at the end. The function $\zeta$, which depends only on 
$\Omega$ and~$\delta$, is taken to be smooth and satisfy
\begin{equation}
 \label{propzeta2}
 |\nabla\zeta|\leq  3/\delta.
\end{equation}

To simplify notation, for $y\in K_\delta$, define
$$
I :=  \int_{\Omega}  |\nabla u|^2 r_y^{-\alpha} \zeta^2 \,dx 
$$
and 
$$
I_r := \int_{\Omega}  u_{r_y}^2 r_y^{-\alpha} \zeta^2 \,dx.
$$
Given $\varepsilon>0$, inequality \eqref{hardyfol} and Cauchy-Schwarz give that
\begin{align*}
\left( (n - 1 - \alpha)I + \alpha I_r \right)^2
&\leq
I  \int_{\Omega\cap\{|\nabla u|>0\}} \big (4|\nabla u|\,\, |\nabla_T(|\nabla u|^{1/2} \zeta) |^2 
+ {H}^2|\nabla u|^2 \zeta^2\big ) 
r_y^{2-\alpha}   \,dx\\
&\hspace{-3.3cm}\leq
I  \int_{\Omega\cap\{|\nabla u|>0\}}  
\big( 4|\nabla u| \, \left| \frac{1}{2}|\nabla u|^{-1/2} (\nabla_T |\nabla u|) \zeta 
+ |\nabla u|^{1/2} \nabla_T\zeta \right|^2  
\\&
\hspace{5.5cm} 
+{H}^2|\nabla u|^2 \zeta^2 \big) r_y^{2-\alpha}   \,dx
\\&\hspace{-3.3cm}\leq
I  \int_{\Omega\cap\{|\nabla u|>0\}}  
\big( (1+\varepsilon) |\nabla_T|\nabla u||^2 + {H}^2|\nabla u|^2 \big) r_y^{2-\alpha}  \zeta^2 \,dx 
\\&\hspace{-1.5cm} 
+ I \, \frac{C}{\varepsilon}  \int_{\Omega\setminus K_{\delta/2}}  
|\nabla u|^2 |\nabla \zeta|^2 r_y^{2-\alpha} \,dx
\\&\hspace{-3.3cm}\leq
I  \int_{\Omega\cap\{|\nabla u|>0\}}   \big( (1+\varepsilon) |\nabla_T|\nabla u||^2 + {H}^2|\nabla u|^2 \big) r_y^{2-\alpha}  
\zeta^2 \,dx 
+ I \, \frac{C}{\varepsilon}  \int_{\Omega\setminus K_{\delta/2}}  
|\nabla u|^2 \,dx;
\end{align*}
in the last inequality we have used \eqref{propzeta2} and that 
\begin{equation}\label{ineqry}
\delta/2 \leq r_y(x)=|x-y|\leq\diam(\Omega)
\qquad\text{for }  x\in \Omega\setminus K_{\delta/2}\text{ and } y\in K_\delta
\end{equation}
---both the lower and the upper bound are needed here, since the sign of $2-\alpha$ will depend on the dimension 
(once we choose $\alpha$). The constant $C$ depends only on $\Omega$ and $\delta$.

Taking into account that $ {H}^2 \leq (n-1)|A|^2$ and $\Omega\setminus K_{\delta/2}\subset
\Omega\setminus K_{\delta}$, we deduce
\begin{align*}
\left( (n - 1 - \alpha)I + \alpha I_r \right)^2&
\\ & \hspace{-4cm}
\leq I  (1+\varepsilon)(n-1) \int_{\Omega\cap\{|\nabla u|>0\}}  \big( |\nabla_T|\nabla u||^2 + |A|^2|\nabla u|^2 \big) r_y^{2-\alpha}  
\zeta^2 \,dx 
+ I \frac{C}{\varepsilon} \Vert \nabla u\Vert^2_{L^2({\Omega\setminus K_{\delta}})}.
\end{align*}
Note that the previous bound is sharp in the radial case since $ {H}^2 = (n-1)|A|^2$ 
and $\nabla_T|\nabla u|=0$ when the level sets are concentric spheres.

We now use the stability condition in the geometric form, estimate \eqref{semi1}, with
$$
\eta = r_y^{(2-\alpha)/2} \zeta.
$$
The following computations must be done with $r_y^{(2-\alpha)/2}$ replaced by a nice regularization 
of it in a small $\kappa$-neighborhood of $y$. Since all terms in the following
computations are given by integrable functions, through dominated convergence we can let
$\kappa\to 0$ to justify the following equalities written directly for $r_y^{(2-\alpha)/2}$.

We deduce that
\begin{align*}
\left( (n - 1 - \alpha)I + \alpha I_r \right)^2  &\\
&\hspace{-3.3cm} \leq  I (1+\varepsilon)(n-1) \int_{\Omega}  |\nabla u|^2 \left | \nabla  \left (r_y^{(2-\alpha)/2} \zeta \right )
\right |^2 \,dx  + I \frac{C}{\varepsilon} \,  \Vert \nabla u\Vert^2_{L^2({\Omega\setminus K_{\delta}})}\\
&\hspace{-3.3cm} \leq  (1+2\varepsilon) (n-1) \frac{(\alpha-2) ^2}{4}  I^2 +
 I \frac{C}{\varepsilon}\int_{\Omega}  |\nabla u|^2 r_y^{2-\alpha} |\nabla\zeta|^2 \,dx +
I\, \frac{C}{\varepsilon} \Vert \nabla u\Vert^2_{L^2({\Omega\setminus K_{\delta}})}\\
&\hspace{-3.3cm} \leq  (1+2\varepsilon) (n-1) \frac{(\alpha-2) ^2}{4}  I^2 +
I\, \frac{C}{\varepsilon} \Vert \nabla u\Vert^2_{L^2({\Omega\setminus K_{\delta}})};
\end{align*}
as before, we have used Cauchy-Schwarz, \eqref{propzeta1}, \eqref{propzeta2}, and \eqref{ineqry}. 
The constant $C$ still depends only on $\Omega$ and $\delta$. 

Let now $\alpha\in [0,n-1)$ satisfy \eqref{condalpha}. We can choose $\varepsilon>0$ to depend only on 
$n$ and $\alpha$ such that $(1+2\varepsilon) (n-1) (\alpha-2)^2/4 < (n - 1 - \alpha)^{2}$. 
Since $(n - 1 - \alpha)^{2}I^{2}\leq ((n - 1 - \alpha)I + \alpha I_r)^2$, our last estimate 
establishes \eqref{estgrad}. We have used again \eqref{propzeta1} and \eqref{ineqry}. 

Finally, if $\Omega$ is a ball centered at the origin, the stability of the solution $u$ leads to its radial
symmetry ---there is no need to assume that $u$ is positive here; see Proposition~ 1.3.4 of~\cite{Dupaigne}. 
As a consequence, with $y=0$, we have $I_{r}=I$ and our last bound gives
$$
(n - 1)^{2}I^2   \leq  (1+2\varepsilon) (n-1) \frac{(\alpha-2) ^2}{4}  I^2 +
I\, \frac{C}{\varepsilon} \Vert \nabla u\Vert^2_{L^2({\Omega\setminus K_{\delta}})}.
$$
Assuming now the weaker condition \eqref{condalpharad} on $\alpha$, we conclude \eqref{estgrad} with $y=0$ as before.
\end{proof}

With this bound in hand, we can now prove our main result.

\begin{proof}[Proof of Theorem \ref{Thm1}]
Assume that an exponent $\alpha$ can be chosen (depending only on $n$) such that
\begin{equation}
 \label{choice}
 n-2<\alpha< n-1 \qquad\text{and}\qquad (\alpha -2)^2/4 < (n-1-\alpha)^2/(n-1).
\end{equation}
Then, by Proposition \ref{weightest},
\begin{equation}
 \label{linfty1}
 \int_{\Omega}  |\nabla u|^2 r_y^{-\alpha} \,dx  \leq C \Vert \nabla u\Vert^2_{L^2(\Omega\setminus K_\delta)}
 \qquad \text{ for all } y\in K_\delta, 
\end{equation}
i.e., such that $\dist(y,\partial\Omega)\geq\delta$.
Here $C$ is a constant depending only on $\Omega$ and~$\delta$.

It is easy to see that \eqref{choice} can be accomplished for some $\alpha$ if $n=3$ or $n=4$. 
For this just take, in both cases, $\alpha=n-2+\varepsilon$ for some small $\varepsilon>0$.

In the radial case, Proposition \ref{weightest} ensures \eqref{linfty1} if there is
an exponent $\alpha$ such that
\begin{equation}
 \label{choiceead}
 n-2<\alpha< n-1 \qquad\text{and}\qquad (\alpha -2)^2/4 < n-1.
\end{equation}
Its existence can be checked to hold if $2\leq n\leq 9$.

In both the nonradial and radial cases, it is simple to deduce the $L^\infty$ bound in $K_\delta$ from 
\eqref{linfty1} and the fact that $\alpha > n-2$. Indeed, take any $y\in K_\delta$ and use 
Lemma~7.16 of \cite{GT} with $\Omega$ in that lemma replaced by the ball $B_{\delta/2}(y)$. We deduce that
\begin{equation}\label{repres}
|u(y) - u_{S_y}| \leq  C \int_{\Omega} |x-y|^{1-n} |\nabla u(x)| \,dx= C \int_{\Omega} r_y^{1-n} |\nabla u|\,dx,
\end{equation}
where 
\begin{equation}\label{aveu}
 u_{S_y}:=\frac{1}{|B_{\delta/2}(y)|} \int_{B_{\delta/2}(y)} u \,dx
\end{equation}
and $C$ is a constant depending only on $n$ and $\delta$.
Now, using \eqref{linfty1}, the Cauchy-Schwarz inequality in \eqref{repres}, and that 
$$
\int_\Omega |x-y|^{2-2n+\alpha} \,dx \leq \frac{|S^{n-1}|}{\alpha-(n-2)} \diam(\Omega)^{\alpha-(n-2)},
$$
from \eqref{repres} and \eqref{aveu} we conclude estimate \eqref{est4}, i.e.,  
\begin{equation}
\label{est4new}
\Vert u\Vert_{L^\infty(K_\delta)} \leq 
C \left( \Vert  u\Vert_{L^1(\Omega)} + \Vert \nabla u\Vert_{L^2(\Omega\setminus K_\delta)}\right)
\end{equation} 
with $C$ depending only on $\Omega$ and $\delta$.

Assume now that $\Omega$ is a convex domain, $u$ solves problem \eqref{problem}, and $u$ is positive in $\Omega$. 
Then, the moving planes method gives the existence of small 
truncated cones of monotonicity for $u$ ---the cones having vertex at any point
in a sufficiently small neighborhood of $\partial\Omega$. This is a result of de Figueiredo-Lions-Nussbaum; 
see Proposition~3.2 of \cite{Cabre}. This leads to the bound 
\begin{equation}
\label{bdryest}
\Vert u\Vert_{L^\infty(\Omega\setminus K_{2\delta})}\leq C\Vert  u\Vert_{L^1(\Omega)}
\end{equation} 
for some constants $\delta>0$ and $C$ depending only on $\Omega$. 
We use this $L^\infty$ bound in $\Omega\setminus K_{2\delta}$ to control $f(u)$ in this set. 
We deduce, by interior and boundary elliptic regularity
for problem \eqref{problem}, 
stronger estimates in the smaller set $\Omega\setminus K_{\delta}$. In particular, we have
$\Vert \nabla u\Vert_{L^2(\Omega\setminus K_{\delta})}\leq C$ for some constant $C$
depending only on $\Omega$, $f$, and $\Vert  u\Vert_{L^1(\Omega)}$. This, combined with \eqref{est4new}, 
gives the desired bound \eqref{est5} since we also have 
$\Vert u\Vert_{L^\infty(\Omega\setminus K_{\delta})}
\leq\Vert u\Vert_{L^\infty(\Omega\setminus K_{2\delta})}\leq C\Vert  u\Vert_{L^1(\Omega)}$ by \eqref{bdryest}.

In the radial case, the argument to prove \eqref{est5} is the same, 
but there is no need to assume the solution $u$ to be positive. Indeed,  Proposition~ 1.3.4 of~\cite{Dupaigne} 
ensures that $u$ is radially symmetric and monotone in the radius $r$. Thus, up to changing $u$ by $-u$, 
we may assume the solution to be positive and radially decreasing. In particular, it suffices to estimate $u(0)$,
and this is done using \eqref{repres} and \eqref{estgrad} (with $y=0$) as in the general case.
\end{proof}

Finally, we give the proof of the $L^p$ estimates in dimension 5.

\begin{proof}[Proof of Corollary \ref{Corol5}]
Let us recall that, for $1\leq p<\infty$ and $0<\lambda\leq n$, the Morrey space $M^{p,\lambda}=M^{p,\lambda}(\R^n)$
is the space of $L^p_{\rm loc}(\R^n)$ functions $w$ for which
\begin{equation}\label{morrey}
\Vert w\Vert^{p}_{M^{p,\lambda}}:= \sup_{y\in\R^n,\, \rho >0} \rho^{\lambda -n}\int_{B_\rho(y)} |w(x)|^p <\infty,
\end{equation}
equipped with this norm.

Assume that $n=5$ 
and extend the functions $u$ and $\nabla u$ by zero outside $\Omega$. Condition \eqref{condalpha}
reads $0\leq \alpha <4$ and $(\alpha-2)^2 <(4-\alpha)^2$, and hence it is satisfied for all $\alpha\in (2,3)$. 
The boundary regularity on convex domains (see the last argument in the proof of Theorem~\ref{Thm1}) 
together with estimate \eqref{estgrad} of Proposition \ref{weightest} give that
$|\nabla u|\in M^{2,\lambda}$ if
$5- \lambda = n -\lambda \in (2,3)$, that is, if $\lambda\in (2,3)$. This result comes with an estimate.

Now, we invoke a result of D. R. Adams (see Proposition~3.1
and Theorems~3.1 and 3.2 in \cite{A}) that gives the following optimal embedding:  
$$
\text{if } 2<\lambda\leq n \text{ and } |\nabla u|\in M^{2,\lambda}, \text{ then }
u\in M^{\frac{2\lambda}{\lambda -2},\lambda}.
$$
In particular, $u\in L^{\frac{2\lambda}{\lambda -2}}(\Omega)$. Letting $\lambda\downarrow 2$, the proof
is completed.

Let us finish mentioning that this $L^{\frac{2\lambda}{\lambda -2}}$ estimate  
can be proved easily using the estimate for $\int_\Omega |\nabla u|^2 r_y^{-\alpha} dx$
without invoking Adams' result. For this, one uses \eqref{repres} and 
a similar argument to that in the proof of Young's inequality for convolutions.
\end{proof}

\section*{Acknowledgments}

\noindent
The author would like to thank the referee for a careful reading and some comments that improved the presentation.

\end{document}